\documentclass[a4paper, 10pt, leqno]{article}

\usepackage[T1]{fontenc}		     
\usepackage[utf8]{inputenc}			 
\usepackage[english]{babel}

\usepackage{amsmath}
\usepackage{amssymb}
\usepackage[numbers]{natbib}

\usepackage{amsthm}
	\theoremstyle{plain}

		\newtheorem{thm}{Theorem}[section]	
			
		\newtheorem{lem}[thm]{Lemma}		
		\newtheorem{prop}[thm]{Proposition}

	\theoremstyle{definition}
		\newtheorem{defn}[thm]{Definition}	

	\theoremstyle{remark}
		\newtheorem{rem}[thm]{Remark}		

\numberwithin{equation}{section}	

\usepackage{mathtools}		
\usepackage{mathrsfs}		
\usepackage{eucal}			

\usepackage{braket}			
\usepackage[all,pdf]{xy}		
\usepackage{mparhack}		
\usepackage{relsize}			

\usepackage{a4wide}		

\usepackage{booktabs}		
\usepackage{multirow}		
\usepackage{caption}		
\usepackage{rotating}
\usepackage{subfig}

\usepackage{varioref}		
\usepackage{footmisc}

\usepackage{graphicx}		
\usepackage{epsfig}

\usepackage{enumerate}		
\usepackage[colorlinks=true, linkcolor=black, citecolor=black,
urlcolor=blue]{hyperref}		
\mathtoolsset{showonlyrefs}
\newcommand{\inpro}[2]{\left\langle #1,#2\right\rangle}

\newcommand{\spfl}[1]{\text{sf}(#1)}

\newcommand{\email}[1]{\href{mailto:#1}{\textsf{#1}}}

\newcommand{\iMor}{\mathrm{m^-}}		

\newcommand{\R}{\mathbb{R}}	

\newcommand{\C}{\mathbb{C}}		

\newcommand{\Lag}[1]{\Lagr({#1})}
\newcommand{\Lagr}{\Lambda}

\newcommand{\Mat}[1]{\mathrm{Mat}(#1,\R)}
\newcommand{\Sym}[1]{\mathrm{Sym}(#1,\R)}

\let\d\relax
\newcommand{\d}{\mathrm{d}}				

\newcommand{\norm}[1]{\left\| #1 \right\|}			
				
\newcommand{\iCLM}{\mu^{\scriptscriptstyle{\mathrm{CLM}}}}

\newcommand{\coiMor}{\mathrm{m^+}}

\newcommand{\noo}[1]{\overset {\mbox{%
\lower1pt\hbox{${\scriptscriptstyle o}$}}}n^{\mbox{%
\lower2pt\hbox{$\scriptscriptstyle #1$}}}}

\DeclareMathOperator{\sgn}{sgn}		
\DeclareMathOperator{\sign}{sign}		

\DeclareMathOperator{\dom}{dom}

\renewcommand{\geq}{\geqslant}
\renewcommand{\hat}{\widehat}

\renewcommand{\=}{\coloneqq}			

\usepackage{bm}

\newcommand{\Id}{I}

\title{Morse Index Theorem for Sturm-Liouville Operators on the Real Line}
\author{ Ran Yang \thanks{The author is supported by  the National Natural Science Foundation of China (No. 12001098, No. 42264007) and the  Natural Science Foundation of Jiangxi Province(No.20232BAB211005)}, \quad Qin Xing\thanks{The author is supported  by  the National Natural Science Foundation of China (No. 12201278) and the Natural Science Foundation of Shandong Province (No. ZR2022QA076)}}
\date{\today}

\date{\today}
\begin{document}
 \maketitle

\begin{abstract}
	The classical Morse index theorem establishes a fundamental connection between the Morse index-the number of negative eigenvalues that characterize key spectral properties of linear self-adjoint differential operators-and the count of corresponding conjugate points. In this paper, we extend these foundational results to the Sturm-Liouville operator on $\R$. In particular, for autonomous Lagrangian systems, we employ a geometric argument to derive a lower bound for the Morse index.  As concrete applications, we  establish a criterion for detecting instability in traveling waves within gradient reaction-diffusion systems.
\end{abstract}


\section{Introduction and description of the problem}

We consider the following $n$-dimensional Sturm-Liouville operator:  
\begin{align}\label{eq:S.L.}
	Lw := -\frac{\d}{\d t}(P(t) \dot{w}(t) + Q(t) w(t)) + Q^\top(t) \dot{w}(t) + R(t) w(t) \text{ with } \dom(L) = W^{2,2}(\R, \R^n).
\end{align}  
Here, the coefficient matrices satisfy  
$ P(t) \in C^1(\R, \Sym{n}) $, $ Q(t) \in C^1(\R, \Mat{n}) $, and $ R(t) \in C(\R, \Sym{n}) $, where $\Mat{n}$ is the set of all $n \times n$ matrices and $\Sym{n}$ is the set of all $n \times n$ symmetric matrices. 

We impose the following assumptions:  
\begin{itemize}  
	\item[(L1)] The matrices $ P(t) $, $ Q(t) $, and $ R(t) $ converge to their respective limits $ P_{\pm} $, $ Q_{\pm} $, and $ R_{\pm} $ as $ t \to \pm \infty $. Furthermore, there exist positive constants $ C_1 $, $ C_2 $, and $ C_3 $ such that  
	\begin{align*}  
		\norm{P(t)} \geq C_1, \quad \norm{Q(t)} \leq C_2, \quad \text{and} \quad \norm{R(t)} \leq C_3 \quad \forall t \in \R.  
	\end{align*}  
	\item[(L2)] The block matrices  
	\[
		\begin{pmatrix}  
			P_- & Q_- \\  
			Q_-^\top & R_-  
		\end{pmatrix}
		\quad \text{and} \quad
		\begin{pmatrix}  
			P_+ & Q_+ \\  
			Q_+^\top & R_+  
		\end{pmatrix}  
	\]   
	are both positive definite.  
\end{itemize}  

The starting point of our study is the classical Sturm-Liouville theory, which asserts that the Morse index of one-dimensional Sturm-Liouville operator equals the number of zeros of the eigenfunction for the first non-negative eigenvalue\cite[Theorem 2.3.3]{kapitula2013spectral}.

However, when $n > 1$, the results of Sturm-Liouville  theorem no longer apply, and determining the Morse index becomes significantly more challenging.

In \cite{AR85}, Arnold observed that, for a system of differential equations, a useful generalization of the notion of ‘oscillation’ (or number of zeros) is given by the Maslov index\cite{Arn67,RS93,CLM94}, a topological invariant counting signed intersections of Lagrangian planes in a symplectic vector space. Increasing attention has been devoted to the Maslov index and its applications to study the stablity of standing pulses, solitary waves  in reaction-diffusion equation \cite{CDB09,CDB09-1,CDB09-2,Jo85} and the linear stability elliptic relative equilibria in $N$-body problem\cite{HO16,BHPT20}. The use of the Maslov index as a tool for determining the Morse index(i.e.,unstable eigenvalue) was pioneered in the works of Jones \cite{Jo85} and Bose \& Jones \cite{BJ95}. Further related studies can be found in the literature, including \cite{CH07,HLS18,HS20,HPWX20}.

Our primary objective is to determine the number of negative eigenvalues of $ L $, i.e., its Morse index. To achieve this, we establish a connection between the Morse index and the Maslov index. Specifically, we demonstrate that the Morse index can be expressed in terms of the Maslov index with  positive direction at every crossing, i.e., the crossing form is positive definite at every crossing. Although the Maslov index is conceptually less elementary than the Morse index, it offers a computational advantage, as it can be determined numerically in a relatively straightforward manner\cite{CDB09,CDB09-2,HO16}.

In the standard symplectic space \( \left( \R^{2n}, \omega \right) \), we denote by \( \Lag{n} \) the Lagrangian Grassmannian, which is the collection of all Lagrangian subspaces of \( \left( \R^{2n}, \omega \right) \). 

By applying a symplectic change of coordinates, the equation \( Lu = 0 \) can be rewritten as the Hamiltonian system:
\begin{align}\label{eq:h.s. eq.}
	\dot{z} = J B(t) z, \quad \text{where} \quad B(t) := \begin{pmatrix}
		P^{-1}(t) & -P^{-1}(t) Q(t) \\
		-Q^\top(t) P^{-1}(t) & Q^T(t) P^{-1}(t) Q(t) - R(t)
	\end{pmatrix}.
\end{align}
Given condition (L1), the limit matrices \( \lim\limits_{t \to +\infty} B(t) := B(+\infty) \)  and \( \lim\limits_{t \to -\infty} B(t) := B(-\infty) \) exist and are both well-defined.
\begin{prop}\cite[Proposition 2.3]{HPWX20}\label{pro:fredholm}
	$L$ is Fredholm if and only if both $JB(+\infty)$ and $JB(-\infty)$ are hyperbolic, meaning that their spectra lie off the imaginary axis.
\end{prop}

Let \( \gamma_{\tau} \) denote the fundamental matrix solution of the linear Hamiltonian system:
\begin{align*}
	\begin{cases}
		\dot{\gamma}_{\tau}(t) = J B(t) \gamma_{\tau}(t), \quad t \in \R, \\
		\gamma_{\tau}(\tau) = I,
	\end{cases}
\end{align*}
which satisfies the semigroup property: \( \gamma_{\sigma}(\xi)\gamma_{\tau}(\sigma) = \gamma_{\tau}(\xi) \) for all \( \xi, \sigma \in \R \).

For \( \tau \in \R \), we define the stable and unstable subspaces as follows:
\begin{align}\label{eq:invariance space}
	E^s(\tau) := \left\{ v \in \R^{2n} \mid \lim_{t \to +\infty} \gamma_{\tau}(t) v = 0 \right\}, \quad E^u(\tau) := \left\{ v \in \R^{2n} \mid \lim_{t \to -\infty} \gamma_{\tau}(t) v = 0 \right\}.
\end{align}
It is well-known (Cfr. \cite{HP17} and references therein) that the path $\tau \mapsto E^s(\tau)$ and $\tau \mapsto E^u(\tau)$ are both Lagrangian and to each ordered pair of Lagrangian paths we can assign an integer known in literature as Maslov index of the pair and denoted by $\iCLM$(Subsection \ref{sec:maslov}). These subspaces satisfy  \( E^{s/u}(\tau) = \gamma_{\sigma}(\tau) E^{s/u}(\sigma) \) for any \( \sigma, \tau \in \R \).

We define the asymptotic stable subspaces at infinity as:
\begin{align*}
	\begin{aligned}
		E^s(+\infty) &:= \left\{ v \in \R^{2n} \mid \lim_{t \to +\infty} \exp(t J B(+\infty)) v = 0 \right\} = V^-(J B(+\infty)), \\
		E^u(-\infty) &:= \left\{ v \in \R^{2n} \mid \lim_{t \to -\infty} \exp(t J B(-\infty)) v = 0 \right\} = V^+(J B(-\infty)).
	\end{aligned}
\end{align*}
Under the condition (L2), by Lemma \ref{lem:hyperbolic_matrix} and \ref{space convergence}, we have that 
\begin{align}
	\lim_{\tau \to +\infty} E^s(\tau) = V^-(JB(+\infty)) \quad \text{and} \quad \lim_{\tau \to -\infty} E^u(\tau) =V^+(JB(-\infty)),
\end{align}
where the convergence is meant in the gap (norm) topology of the Lagrangian Grassmannian.

In \cite{HPWX20}, under condition (L2), the authors establish a fundamental relationship between the Morse index \(\iMor(L)\) and the Maslov index \(\iCLM(E^s(\tau),E^u(-\tau);\tau\in[0,+\infty))\), given by
\[
\iMor(L)=-\iCLM(E^s(\tau),E^u(-\tau);\tau\in[0,+\infty)).
\]
For n > 1, computing the Maslov index becomes challenging without additional knowledge of the n-1 solutions. To address this, \cite[Theorem 4.1]{BCJLM18} considers the second-order linear operator
\[
H = -\frac{\d^2}{\d t^2} + A(t), \quad \text{with } \dom H = W^{2,2}(\mathbb{R},\mathbb{R}^n).
\]
Under condition (L1), the authors employ the Evans function and Rouché's theorem to establish the result
\[
\iMor(H) = \sum_{\tau \in \mathbb{R}} \dim(E^u(\tau) \cap \Lambda_D),
\]
where \[
\Lambda_D=\{(u,0)^\top |u\in\R^n\}
\]is the horizontal (Dirichlet) Lagrangian subspace.

Inspired by this result, we adopt a different methodology to extend these results to general Sturm-Liouville operators.
\begin{thm}\label{thm:main result}
	Under the assumptions (L1) and (L2),  there exists $T_\infty\in \R$ such that, for all $T>T_\infty$,  we have:
	$$
		\iMor(L) = \iCLM(\Lambda_D,E^u(\tau);\tau\in(-\infty,T])=\sum_{\tau \in  \R} \dim(E^u(\tau) \cap \Lambda_D).
	$$
\end{thm}

\begin{rem}
	Let $L: \mathbb{R}^n\times \mathbb{R}^n\rightarrow \mathbb{R}$ be a smooth autonomous Lagrangian function satisfying the following Legendre convexity condition:  
	
	\noindent  $L$ is $C^2$-convex, meaning that the quadratic form  
	\begin{align*}
	\left\|D_{v v}^2 L(q, v)\right\| \geq \ell_0 I>0, \quad \forall(q, v) \in \mathbb{R}^n\times \mathbb{R}^n.
	\end{align*}
	
	In what follows, we denote by $u^{-}, u^{+} \in \mathbb{R}^n$ two rest points of the Lagrangian vector field $\nabla L$. That is, they satisfy  
	\[
	\nabla L\left( u^{ \pm}, 0\right)=0.
	\]
	
	A heteroclinic orbit $u^*$ asymptotic to $u^{ \pm}$ is a $C^2$-solution of the following boundary value problem:  
	\begin{align}\label{eq:lagrangian sys.}
	\begin{cases}
	\frac{\d}{\d t} \partial_v L(t, u(t), \dot{u}(t))=\partial_q L(t, u(t), \dot{u}(t)), \quad t \in \mathbb{R}, \\
	\lim \limits_{t \rightarrow-\infty} u(t)=u^{-}, \quad \lim \limits_{t \rightarrow+\infty} u(t)=u^{+}.
	\end{cases}
	\end{align}
	
	By linearizing \eqref{eq:lagrangian sys.} along $u^*$, we obtain the following Sturm-Liouville operator:
	\[
	Lw := -\frac{\d}{\d t}(P(t) \dot{w}(t) + Q(t) w(t)) + Q^\top(t) \dot{w}(t) + R(t) w(t), \quad \text{with } \dom(L) = W^{2,2}(\mathbb{R}, \mathbb{R}^n),
	\]
	where  
	\[
	P(t):=\partial_{v v} L( u(t), \dot{u}(t)), \quad Q(t):=\partial_{u v} L(u(t), \dot{u}(t)), \quad R(t):=\partial_{u u} L(u(t), \dot{u}(t)).
	\]
	
	Since \eqref{eq:lagrangian sys.} is autonomous, it follows that $\dot{u}^*(t)\in \ker (L)$. Moreover, the vector  
	\[
	(P(\tau)\ddot{u}^*(\tau)+Q(\tau)\dot{u}^*(\tau),\dot{u}^*(\tau))^\top \in E^u(\tau).
	\]
	By applying Theorem \ref{thm:main result}, we conclude that the number of critical points of $u^*$ (i.e., points where $\dot{u}^*(t) = 0$) serves as a lower bound for the Morse index of the operator $L$.
\end{rem}

\section{The Proof of Theorem \ref{thm:main result}}

In this section, we firstly recall known results that provide sufficient conditions for the hyperbolicity of the Hamiltonian matrix $J B$, given by
\begin{align}\label{eq:hamiltonian_matrix}
B=\begin{pmatrix}
P^{-1} & -P^{-1} Q \\
-Q^{\top} P^{-1} & Q^{\top} P^{-1} Q - R
\end{pmatrix},
\end{align}
based on the non-vanishing determinant of a suitably chosen matrix.

The following results, which are well established in the literature\cite{HPWX20}, will be useful in our analysis.

\begin{lem}\cite[Corollary C.2]{HPWX20}\label{lem:hyperbolic_matrix}
If the matrix
$\begin{pmatrix} P & Q \\ Q^{\top} & R \end{pmatrix}$
is positive definite, then the Hamiltonian matrix $J B$ is hyperbolic.
\end{lem}

\begin{lem}\cite[Lemma C.6]{HPWX20}\label{lem:transversal_L_D}
If the matrix
$\begin{pmatrix} P & Q \\ Q^{\top} & R \end{pmatrix}$
is positive definite, then we have that $V^{\pm}\left(J B \right) \cap L_D = \{0\}.$
\end{lem}

Fixed \( \xi \geq 0 \), we define the operator \( L_\xi := L + \xi \Id \), where \( \dom(L_\xi) := W^{2,2}(\R, \R^n) \).

Under condition (L2), a straightforward calculation shows that the matrices
\begin{align*}
\begin{pmatrix} P_- & Q_- \\ Q_-^{\top} & R_- + \xi \Id \end{pmatrix} \quad \text{and} \quad 
\begin{pmatrix} P_+ & Q_+ \\ Q_+^{\top} & R_+ + \xi \Id \end{pmatrix}
\end{align*}
are both positive definite for all \( \xi \geq 0\). Consequently, by Lemma \ref{lem:hyperbolic_matrix} and Proposition \ref{pro:fredholm}, we conclude that \( L_\xi \) is a Fredholm operator.

\begin{lem}\label{lem:nondegenerate_L}
	Under the assumptions (L1) and (L2), the operator \( L_\xi \) is non-degenerate for every \( \xi \geq \frac{8C_2^2}{C_1} + C_3 \), where \( C_1, C_2 \), and \( C_3 \) are as defined in (L1).
	\end{lem}
	
	\begin{proof}
		Suppose that \( w \in \ker(L_\xi) \). Then, applying the Cauchy-Schwarz inequality, we obtain
		\begin{align}\label{eq:nondegenerate_L}
			\begin{split}
				0 &= \int_{-\infty}^{+\infty}\inpro{-\frac{\d}{\d t}(P(t) \dot{w}(t) + Q^\top(t) w(t)) + Q(t) \dot{w}(t) + R(t) w(t) + \xi w}{w} \d t \\
				&= \int_{-\infty}^{+\infty}\inpro{P(t) \dot{w}(t)}{\dot{w}(t)} \d t + \int_{-\infty}^{+\infty}\inpro{Q(t) \dot{w}(t)}{w(t)} \d t \\
				&\quad + \int_{-\infty}^{+\infty}\inpro{Q^{\top}(t) w(t)}{\dot{w}(t)} \d t + \int_{-\infty}^{+\infty}\inpro{R(t) w(t) + \xi w(t)}{w(t)} \d t \\
				&\geq C_1 \int_{-\infty}^{+\infty} |\dot{w}(t)|^2 \d t - 2 C_2 \int_{-\infty}^{+\infty} \left( \delta |\dot{w}(t)|^2 + \frac{1}{\delta} |w(t)|^2 \right) \d t + (\xi - C_3) \int_{-\infty}^{+\infty} |w(t)|^2 \d t \\
				&= \left( C_1 - 2\delta C_2 \right) \int_{-\infty}^{+\infty} |\dot{w}(t)|^2 \d t + \left( \xi - C_3 - \frac{2}{\delta} C_2 \right) \int_{-\infty}^{+\infty} |w(t)|^2 \d t.
				\end{split}
			\end{align}
		Setting \( \delta := \frac{C_1}{4C_2} \), it follows from \eqref{eq:nondegenerate_L} that if \( \xi \geq \frac{8C_2^2}{C_1} + C_3 \), then \( L_\xi \) is non-degenerate.
		\end{proof}

		Since \( L \) is a Fredholm operator, we define a family of operators 
		\begin{align}\label{eq:operator path}
			L_\lambda = L + \lambda \epsilon \Id 
		\end{align}
		where \( \epsilon > 0 \) and \( \lambda \in [0,1] \), with \( \epsilon \) being small enough to satisfy the following condition:
		
		\begin{itemize}
			\item[(L3)] \( L_\lambda \) is non-degenerate for all \( \lambda \in (0,1] \).
		\end{itemize}

		\begin{lem}\label{lem:morse 0= morse 1}
			Under the assumptions (L1)-(L3), for the operator path $L_\lambda $ in \eqref{eq:operator path}, we have that 
			\[
			\iMor(L) = \iMor(L_1).
			\]
		\end{lem}
		
		\begin{proof}
			Let \( \hat{\xi} > \frac{8C_2^2}{C_1} + C_3 \). We construct the following homotopy operator path:
			\[
			L_{\xi,\lambda} = L + \xi \Id + \lambda \epsilon \Id : (\xi, \lambda) \in [0, \hat{\xi}] \times [0, 1],
			\]
			By the homotopy invariance of the spectral flow, we have:
			\begin{align}\label{eq:morse equivalent}
				\spfl{L_{\xi,0};\xi\in[0,\hat\xi]} + \spfl{L_{\hat\xi,\lambda};\lambda\in[0,1]} = \spfl{L_{\xi,1};\xi\in[0,\hat\xi]} + \spfl{L_{0,\lambda};\lambda\in[0,1]}.
			\end{align}
			By \eqref{eq:compute sf by crossing form}, \eqref{eq:def of cross form in sf}, and Lemma \ref{lem:nondegenerate_L}, we have that:
			\[
				\spfl{L_{\xi,0};\xi\in[0,\hat\xi]} = \iMor(L),
			\text{ and }		\spfl{L_{\xi,1};\xi\in[0,\hat\xi]} = \iMor(L_1).
			\]
			By \eqref{eq:compute sf by crossing form}, \eqref{eq:def of cross form in sf}, and Lemma \ref{lem:nondegenerate_L} along with condition (L3), we have:
			\[
				\spfl{L_{\hat\xi,\lambda};\lambda\in[0,1]} = \spfl{L_{0,\lambda};\lambda\in[0,1]} = 0.
			\]
			Therefore, by the above discussion and \eqref{eq:morse equivalent}, we conclude the proof.
		\end{proof}

By applying a symplectic change of coordinates, the family of equations \( L_\lambda u = 0 \) (\( \lambda \in [0,1] \)) can be rewritten as the following family of Hamiltonian systems:
\begin{align}\label{eq:family h.s.}
	\dot{z} = J B_\lambda(t) z \quad \text{where} \quad B_\lambda(t) := \left(\begin{array}{cc}
	P^{-1}(t) & -P^{-1}(t) Q(t) \\
	-Q^{\top}(t) P^{-1}(t) & Q^{\top}(t) P^{-1}(t) Q(t) - R(t) - \lambda\epsilon \Id
	\end{array}\right).
\end{align}

Given condition (L1), the limit matrices \( \lim\limits_{t \to +\infty} B_\lambda(t) := B_\lambda(+\infty) \) and \( \lim\limits_{t \to -\infty} B_\lambda(t) := B_\lambda(-\infty) \) exist and are well-defined.

Similarly, in  \eqref{eq:invariance space}, we denote \( E_\lambda^s(\tau) \) and \( E_\lambda^u(\tau) \) as the stable and unstable spaces for the linear Hamiltonian systems \eqref{eq:family h.s.}, respectively. Moreover, for simplicity, we omit the subscript $0$ when $\lambda=0$.

Under condition (L2), we have that the matrices
\[
\begin{pmatrix}
	P_- & Q_- \\
	Q_-^{\top} & R_- + \epsilon\lambda \Id
\end{pmatrix}
\quad \text{and} \quad
\begin{pmatrix}
	P_+ & Q_+ \\
	Q_+^{\top} & R_+ + \epsilon \lambda\Id
\end{pmatrix}
\]
are both positive definite for all \( \lambda \in [0,1] \). By Lemma \ref{lem:hyperbolic_matrix}, we conclude that \( J B_\lambda(\pm\infty) \) are both hyperbolic for all \( \lambda \in [0,1] \). Therefore, \( E_\lambda^{s/u}(\tau) \in \Lag{n} \) for all \( \tau \in \mathbb{R} \).

\begin{lem}\cite[Theorem 2.1]{AM03}\label{space convergence}
	Under Condition (L2), for each \( \lambda \in [0,1] \) the following holds:
	\begin{itemize}
		\item[(i)] The stable and unstable subspaces satisfy
			$$
			\lim_{\tau \to +\infty} E_\lambda^s(\tau) = V^-(JB_\lambda(+\infty)) \quad \text{and} \quad \lim_{\tau \to -\infty} E_\lambda^u(\tau) = V^+(JB_\lambda(-\infty))
			$$
			in the gap metric topology of \( \Lag{n} \).
			
		\item[(ii)] For any complementary subspace \( W \subset \R^{2n} \) to \( E_\lambda^s(\tau) \) (resp. \( E_\lambda^u(\tau) \)):
					$$
					\gamma_{\tau,\lambda}(\sigma)W \to V^{+}(J B_\lambda(+\infty)) \quad (\text{resp.} \ V^{-}(J B_\lambda(-\infty))),
					$$
					where \( \gamma_{\tau,\lambda}(\sigma) \) denotes the fundamental matrix solution for the linear Hamiltonian system \eqref{eq:family h.s.}.
			\end{itemize}
\end{lem}

We notice that the matrices 
$
\begin{pmatrix}
	P_- & Q_- \\
	Q_-^{\top} & R_- + \epsilon \Id
\end{pmatrix}
\quad \text{and} \quad
\begin{pmatrix}
	P_+ & Q_+ \\
	Q_+^{\top} & R_+ + \epsilon \Id
\end{pmatrix}
$
are both positive definite. By Lemma \ref{lem:transversal_L_D}, we have that \( V^{\pm}(JB_1(+\infty)) \cap \Lambda_D=\{0\} \) and \( V^{\pm}(JB_1(-\infty)) \cap \Lambda_D =\{0\}\). Denote by
$
\begin{pmatrix}
	M_1^{\pm} \\ I
\end{pmatrix}
\quad \text{and} \quad
\begin{pmatrix}
	N^{\pm}_1 \\ I
\end{pmatrix}
$
the Lagrangian frames of \( V^{\pm}(JB_1(-\infty)) \) and \( V^{\pm}(JB_1(+\infty)) \), respectively where $M^{\pm}_1, N^{\pm}_1\in \operatorname{Sym}(n,\R)$.

\begin{lem}\cite[Lemma 3.9]{HPWX20}\label{lem:matrix definite}
	Under the assumption (L2), then \( M^+_1, N^+_1 \) are positive definite and \( M_1^-,N^{-}_1 \) are negative definite.
\end{lem}
From Lemma \ref{lem:matrix definite}, it is easy to see that 
\[
	V^+(JB_1(-\infty)) \cap V^-(JB_1(+\infty))=\{0\},
\]
From Lemma \ref{space convergence}, we have that/ \( \lim\limits_{\tau \to +\infty} E_1^s(\tau) = V^-(J B_1(+\infty)) \) and \( \lim\limits_{\tau \to -\infty} E_1^u(\tau) = V^+(J B_1(-\infty)) \), there exists a time \( T_0 > 0 \) such that
\begin{align}\label{eq:constant $T_0$}
	E_1^s(\tau_1) \cap E^u_1(\tau_2)\text{ for all } \tau_1>T_0 \text{ and }\tau_2< -T_0.
\end{align}
From \cite[Theorem 2]{HPWX20}, we have the following relation:
\begin{align}\label{eq:morse index 1= maslov index 1}
	\iMor(L_1) = -\iCLM(E_1^s(\tau), E_1^u(-\tau); \tau\in[0, +\infty)).
\end{align}
so we have that 
\begin{lem}\label{lem:Maslov equivalent}
	For $T>T_0$, we have that 
	\begin{align*}
		\iCLM\left(E_1^s(\tau), E_1^u(-\tau) ; \tau\in [0,+\infty)\right) = -\iCLM(E_1^s(T), E_1^u(\tau); \tau\in(-\infty, T]).
	\end{align*}
\end{lem}

\begin{proof}
 Let \( T >T_0 \), we define the following homotopy Lagrangian path:
\begin{align*}
\left(E_1^s(\tau + s T), E_1^u(-\tau + s T)\right), \quad (\tau, s) \in [0,+\infty)\times [0, 1].
\end{align*}
It is important to note that \( \dim\left(E_1^s(s T) \cap E_1^u(s T)\right) \) remains constant for all \( s \in [0, 1] \), and  \( E_1^s(+\infty) \cap E_1^u(-\infty)=\{0\} \).

By the stratum homotopy invariance property, the reversal property of the Maslov index, we conclude that
\begin{align}\label{eq:sf=maslov 1}
\begin{aligned}
& \iCLM\left(E_1^s(\tau), E_1^u(-\tau) ; \tau \in [0,+\infty)\right) \\=& \iCLM\left(E_1^s(\tau + T), E_1^u(-\tau + T) ; \tau \in [0,+\infty)\right) \\
= & \iCLM\left(E_1^s(\tau + 2T), E_1^u(-\tau) ; \tau \in [-T, +\infty)\right)\\=&-\iCLM\left(E_1^s(-\tau + 2T), E_1^u(\tau) ; \tau \in (-\infty,T]\right)
\end{aligned}
\end{align}
By \eqref{eq:constant $T_0$}, we have that $E_1^s(-\tau + 2T)\cap E_1^u(\tau)=\{0\}$ for all $\tau<-T_0$, then by \eqref{eq:sf=maslov 1} we have that 
\begin{align}\label{eq:sf=maslov 1*}
	\iCLM\left(E_1^s(\tau), E_1^u(-\tau) ; \tau \in [0,+\infty)\right)=-\iCLM\left(E_1^s(-\tau + 2T), E_1^u(\tau) ; \tau \in [-T_0,T]\right)
\end{align}
 We construct the following homotopy Lagrangian path:
\begin{align*}
\left(E_1^s\left(T + s \left(T - \tau\right)\right), E_1^u(\tau)\right), \quad (\tau, s) \in \left[-T_0, T\right] \times [0, 1].
\end{align*}

By the stratum homotopy invariance, we deduce that
\begin{align}
& \iCLM \left(E_1^s\left(-\tau + 2T\right), E_1^u(\tau) ; \tau \in \left[-T_0, T\right]\right) \\
=&\iCLM\left(E_1^s\left(T\right), E_1^u(\tau) ; \tau \in \left[-T_0, T\right]\right) \\
=&\iCLM\left(E_1^s\left(T\right), E_1^u(\tau) ; \tau \in \left(-\infty, T\right]\right).
\end{align}
This, together with Equation \eqref{eq:sf=maslov 1*}, completes the proof.
\end{proof}

Thus, by \eqref{eq:morse index 1= maslov index 1} and Lemma \ref{lem:Maslov equivalent}, we have
\begin{align}\label{eq:morse 1= maslov 1*}
	\iMor(L_1) = \iCLM(E_1^s(T), E_1^u(\tau); \tau\in(-\infty, T]),
\end{align}
where \( T > T_0 \).

We begin by introducing the following family of differential operators:
\[
L_{T,\lambda}w = -\frac{\mathrm{d}}{\mathrm{d}t}\big(P(t) \dot{w}(t) + Q(t) w(t)\big) + Q^\top\!(t) \dot{w}(t) + R(t) w(t) + \lambda \epsilon w, \quad t \in (-\infty, T],
\]
where the domain is specified as \( \dom(L_{T,\lambda}) = \left\{ u \in W^{2,2}((-\infty, T], \mathbb{R}^n) \mid u(T) = 0 \right\} \).

For analytical purposes, we introduce the associated Hamiltonian system:
\begin{align}\label{eq:trunct h.s.}
	\begin{cases}
		\dot{z}(t) = J B_\lambda(t) z(t), & t \in (-\infty, T] \\
		\lim\limits_{t \to -\infty} z(t) = 0, & z(T) \in \Lambda_D
	\end{cases}
\end{align}
These Hamiltonian systems are closely linked to the following family of boundary value problems:
\begin{align}
	\begin{cases}
		L_{T,\lambda} u(t) = 0 \\
		\lim\limits_{t \to -\infty} u(t) = u(T) = 0
	\end{cases}
\end{align}
{
\begin{lem}\label{lem:trunct Maslov=Maslov}
	Under the assumptions (L1)-(L3), there exists $T_1>0$ such that for all $T\geq T_1$, 
	\begin{align}
		\iCLM(\Lambda_D, E^s(\tau); \tau\in(-\infty, T]) = \iCLM(\Lambda_D, E_1^s(\tau); \tau\in(-\infty, T]).
	\end{align}
\end{lem}

\begin{proof}
	Consider the Maslov index
\[
	\iCLM(\Lambda_D, E_\lambda^s(\tau); \tau\in(-\infty, T]).
\]
For a crossing point $\tau_0$ (i.e., $\Lambda_D\cap E^u_\lambda(\tau_0)\neq\{0\}$), a straightforward calculation reveals that the crossing form satisfies
\[
\Gamma(E^u_\lambda(\tau),\Lambda_D;\tau_0)=\langle P(\tau_0)v,v\rangle >0
\]
for $\begin{pmatrix}
	v\\0
\end{pmatrix}\in \Lambda_D\cap E^u_\lambda(\tau_0)$. Applying \eqref{eq:compute maslov by cross form}, we obtain
\begin{align}
	\iCLM(\Lambda_D, E_\lambda^s(\tau); \tau\in(-\infty, T])=\sum_{\tau\in(-\infty,T)}\dim (\Lambda_D\cap E^u_\lambda(\tau)),
\end{align}
demonstrating that $\iCLM(\Lambda_D, E_\lambda^s(\tau);\tau\in (-\infty, T])$ is non-decreasing in $T$.

Given $E^u_\lambda(-\infty)\cap \Lambda_D=\{0\}$, direct computation yields
\begin{align}
	&\iCLM(\Lambda_D, E_\lambda^s(\tau); (-\infty, T])-\iCLM(E^s_\lambda(T), E_\lambda^s(\tau); (-\infty, T])\\
	=&s(E^u_\lambda(-\infty),E^u_\lambda(T);E^s(T),\Lambda_D) \quad (\text{by \eqref{eq:def of Hormand}})\\
	=&\iota(E^u_\lambda(-\infty),E^u_\lambda(T),E_\lambda^s(T))-\iota(E^u_\lambda(-\infty),E^u_\lambda(T),\Lambda_D) \quad (\text{by \eqref{eq:relationship Hormand trip}})\\
	=&\coiMor(\mathscr{Q}(E^u_\lambda(-\infty),E^u_\lambda(T),E_\lambda^s(T)))-\coiMor(\mathscr{Q}(E^u_\lambda(-\infty),E^u_\lambda(T),\Lambda_D)) \quad (\text{by \eqref{eq:def of trip 2}})\\
	\leq & \coiMor(\mathscr{Q}(E^u_\lambda(-\infty),E^u_\lambda(T),E_\lambda^s(T)))\leq n \quad (\text{by \eqref{eq:definition q}})
\end{align}
This establishes the inequality
\[
	\iCLM(\Lambda_D, E_\lambda^s(\tau); \tau\in(-\infty, T])\leq \iCLM(E^s_\lambda(T), E_\lambda^s(\tau); \tau\in(-\infty, T])+n.
\]
For sufficiently large $T$, \cite[Theorem 2]{CH07} implies
\[
	\iCLM(E^s_\lambda(T), E_\lambda^s(\tau); \tau\in(-\infty, T])=\iMor(L+\lambda\epsilon\Id)\leq \iMor(L),
\]
leading to the uniform bound
\[
	\iCLM(\Lambda_D, E_\lambda^s(\tau); \tau\in(-\infty, T])\leq  \iMor(L)+n
\]
for all $\lambda\in[0,1]$. Consequently, for each $\lambda\in[0,1]$, there exist at most $\iMor(L)+n$ points where $\Lambda_D\cap E^u_\lambda(\tau)\neq \{0\}.$ In particular, for $\lambda=0,1$, there exists $ T_1>0$ such that $\Lambda_D\cap E^u(\tau) =\{0\}$ and $\Lambda_D\cap E^u_1(\tau)=\{0\}$ hold for all $\tau\geq T_1$.

So there must exists $T>T_1$ with $\Lambda_D\cap E^u_\lambda(T)=\{0\}$ for all $\lambda\in[0,1]$, system \eqref{eq:trunct h.s.} admits only trivial solutions. Let $ F_{T,\lambda}$ denote the corresponding Hamiltonian operator of \eqref{eq:trunct h.s.}, then
\begin{align}\label{eq:trunct Maslov=Maslov1}
	\spfl{ F_{T,\lambda};\lambda\in[0,1]}=0.
\end{align}

Combining \cite[Theorem 2]{HP17} with \eqref{eq:trunct Maslov=Maslov1} yields
\begin{align}\label{eq:trunct Maslov=Maslov2}
		&\iCLM(\Lambda_D, E^s(\tau); \tau\in(-\infty, T]) \\
		= &\iCLM(\Lambda_D, E_1^s(\tau); \tau\in(-\infty, T]) + \spfl{F_{T,\lambda}; \lambda\in [0,1]} \\
		=& \iCLM(\Lambda_D, E_1^s(\tau); \tau\in(-\infty, T]).
\end{align}
The invariance of both $\iCLM(\Lambda_D, E^s(\tau);\tau\in (-\infty, T])$ and $\iCLM(\Lambda_D, E_1^s(\tau);\tau\in (-\infty, T]) $ for $T\geq T_1$ ensures that \eqref{eq:trunct Maslov=Maslov2} holds universally for $T\geq T_1$.
\end{proof}

\begin{proof}[The proof of Theorem \ref{thm:main result}]
	
	Recall the triple index and Hörmander index defined in Section 4.2, then  we have:
	\begin{align}\label{eq:difference of maslov}
		\begin{split}
			&\iCLM(E_1^s(T), E_1^u(\tau); \tau\in(-\infty, T]) - \iCLM(\Lambda_D, E_1^u(\tau); \tau\in(-\infty, T]) \\
			=& s(E_1^u(-\infty), E_1^u(T); \Lambda_D, E_1^s(T)) \quad (\text{by \eqref{eq:def of Hormand}})\\
					\end{split}
	\end{align}

By \eqref{eq:morse 1= maslov 1*}, we have that $\iCLM(E_1^s(T), E_1^u(\tau); \tau\in(-\infty, T])$ is invariant for all $T>T_0$.

	Since \( L_1 \) is non-degenerate, we have \( E_1^s(0) \cap E_1^u(0)=\{0\} \). By (ii) of Lemma \ref{space convergence}, we then have
	\[
	\lim_{\tau \to +\infty} E_1^u(\tau) = \lim_{\tau \to +\infty} \gamma_{0, 1}(\tau) E_1^u(0) = V^+(JB_1(+\infty)).
			\]
			Moreover, if $\tau\geq T_1$, we have that 
\[
		\Lambda_D\cap E_1^u(\tau)=\{
			0
		\}.
		\]
This implies that $\iCLM(L_D, E_1^u(\tau);\tau\in (-\infty, T])$ is invariant for all $T>T_1$.

Set $T_\infty\=\max\{T_0,T_1\}$, so \eqref{eq:difference of maslov} holds for all $T>T_\infty$.

So we have that 
\begin{align}
			&\lim_{T\to+\infty}s(E_1^u(-\infty), E_1^u(T); \Lambda_D, E_1^s(T)) \\
		=&s(E_1^u(-\infty),V^+(JB_1(+\infty);\Lambda_D,V^-(JB_1(+\infty))))\\
		=& \iota(E_1^u(-\infty), \Lambda_D, V^-(JB_1(+\infty))) - \iota(V^+(JB_1(+\infty)), \Lambda_D, V^-(JB_1(+\infty))) \quad (\text{by \eqref{eq:relationship Hormand trip}})\\
   =& \coiMor(\mathscr{Q}(E_1^u(-\infty), \Lambda_D, V^-(JB_1(+\infty)))) - \coiMor(\mathscr{Q}(V^+(JB_1(+\infty)), \Lambda_D, V^-(JB_1(+\infty))))\quad(\text{by \eqref{eq:def of trip 2}}) \\
   =& \coiMor(M_1^+ - N^-_1) - \coiMor(N_1^+-N_1^-)   =n - n=0\quad (\text{by  \eqref{eq:definition q} and Lemma \ref{lem:matrix definite} })
   \end{align}
so we conclude that
	\[
	\iCLM(E_1^s(T), E_1^u(\tau); \tau\in(-\infty, T]) = \iCLM(\Lambda_D, E_1^u(\tau); \tau\in(-\infty, T]).
	\]
	Thus, by \eqref{eq:morse 1= maslov 1*} and  lem \ref{lem:trunct Maslov=Maslov}, we have that
	\begin{align}
					&\iMor(L) = \iMor(L_1) = \iCLM(E_1^s(T), E_1^u(\tau); \tau\in(-\infty, T]) \\
			=& \iCLM(\Lambda_D, E_1^u(\tau); \tau\in(-\infty, T]) = \iCLM(\Lambda_D, E^u(\tau); \tau\in(-\infty, T]),
			\end{align}
	i.e., 
	\begin{align}\label{eq:morse = maslov }
		\iMor(L) = \iCLM(\Lambda_D, E^u(\tau); \tau\in(-\infty, T]),
	\end{align}
	where $T>T_\infty$.

	Let \( \tau_0 \) be a crossing instant (i.e., \( \Lambda_D \cap E^u(\tau_0) \neq \{0\} \)). By \eqref{eq:def of crossing form in Maslov}, the associated crossing form satisfies  
	$$
		\Gamma(E^u(\tau), \Lambda_D; \tau_0) = \langle B(\tau_0) \xi, \xi \rangle = \langle P^{-1}(\tau_0) v, v \rangle,
	$$  
	where \( \xi = \begin{pmatrix} v \\ 0 \end{pmatrix} \in \Lambda_D \cap E^u(\tau_0) \). So from \eqref{eq:morse = maslov } and \eqref{eq:compute maslov by cross form} with the above discussion, we have that 
	\[
		\iMor(L) = \sum_{\tau\in\R} \dim(E^u(\tau) \cap \Lambda_D).
	\] 

\end{proof}

}
\section{Instability of  Traveling Solutions in Gradient Reaction–Diffusion Systems}
Consider the reaction-diffusion system:
\begin{align}\label{eq:r.d.eq.}
u_t= u_{x x}+\nabla F(u), \quad u \in \R^n,
\end{align}
where $x, t \in \R$ denote space and time, respectively, $u \in \R^n$, and $\nabla F$ represents the gradient of the function $F: \R^n \rightarrow \R$. A traveling wave solution $u^*$, which depends on the moving frame $\xi=x-c t$, satisfies equation \eqref{eq:r.d.eq.} along with the asymptotic condition:
\begin{align*}
u^*(\xi)\to u_\pm \text{ as } \xi \rightarrow \pm\infty.
\end{align*}
Here, $u_{\pm}$ denote the constant equilibria of \eqref{eq:r.d.eq.}, satisfying $\nabla F\left(u_{\pm}\right)=0$.

Rewriting in the moving frame $\xi=x-c t$, a traveling front solution $u^*$ of \eqref{eq:r.d.eq.} can be regarded as a homoclinic solution $w^*(\xi)=u^*(x,t)$ of the following equation:
\begin{align}\label{eq:station eq.}
\begin{cases}
	w_{\xi \xi}+c w_{\xi}+ \nabla F(w)=0, \\
	\lim\limits_{\xi \rightarrow -\infty} w(\xi)=u_-\text{ and }\lim \limits_{\xi \rightarrow +\infty} w(\xi)=u_+.
\end{cases}
\end{align}

The stability analysis is closely related to the spectral properties of the operator:
\[
L:=\frac{\d^2}{\d\xi^2}+c \frac{\d}{\d \xi}+B(\xi),
\]
which arises from linearizing \eqref{eq:station eq.} along $w^*$, where $B(\xi)=\nabla^2 F\left(w^*\right)$. The matrices $B({\pm\infty}):=\lim \limits_{\xi \rightarrow \pm\infty} B(\xi)$ are well-defined, and there exists a constant $C>0$ such that:
\[
\langle B(\xi) v, v\rangle \leqslant C|v|^2, \quad \forall (\xi, v) \in \R \times \R^n.
\]

Since a wave solution of \eqref{eq:r.d.eq.} possesses translation invariance, it is termed nondegenerate if zero is a simple eigenvalue of $L$.

\begin{defn}
	A nondegenerate wave solution of \eqref{eq:r.d.eq.} is spectrally stable if all nonzero eigenvalues of $L$ are negative.
\end{defn}

In this paper, we consider the following assumption:
\begin{itemize}
	\item[(H)] The matrices $B(\pm\infty)$ are both negative definite. 
\end{itemize}

Due to the presence of the $\frac{\d}{\d \xi}$ term, $L$ does not exhibit a Hamiltonian structure. This issue can be circumvented by considering the transformed operator:
\begin{align*}
\mathbb{L}:=-e^{\frac{c \xi}{2}} L e^{-\frac{c \xi}{2}}=-\frac{{\d}^2}{\d\xi^2}+\frac{c^2}{4} \Id-B(\xi),
\end{align*}
as discussed in \cite{HS18,C19}.

Under assumption (H), conditions (L1) and (L2) hold for $\mathbb{L}$. By applying a symplectic change of coordinates, the equation $\mathbb{L}u=0$ can be rewritten as the Hamiltonian system:
\begin{align}\label{eq:hs about rd}
	\dot{y}=JA(t)y,
\end{align}
where
\[
A(t)=\begin{pmatrix}
	\Id&0\\0&B-\frac{c^2}{4}\Id
\end{pmatrix}.
\]
Denoting $E^u(\tau)$ as the unstable space of \eqref{eq:hs about rd}, Theorem \ref{thm:main result} gives:
\begin{align}\label{eq:morse index about r.d.eq.}
	\iMor(\mathbb{L})=\sum_{\tau\in \R}\dim \left(\Lambda_D\cap E^u(\tau)\right).
\end{align}

In \cite{X21}, the author considers the eigenvalue problem of the operator:
\[
\hat L:=\frac{\d^2}{\d\xi^2}+c \frac{\d}{\d \xi}+QB(\xi),
\]
where $Q=\begin{pmatrix}
	\Id_r&0\\0&-\Id_{n-r}
\end{pmatrix}$, and establishes the following result:

\begin{prop}\cite[Corollary 2.5.]{X21}\label{pro:equivalent eigenvalue}
	For $\lambda \in \C\backslash\C^-$ and $\phi \in W^{2,2}\left(\R, \C^n\right)$, we have:
	\[
	\phi \in \ker(\hat L-\lambda \Id) \quad \text{if and only if} \quad e^{\frac{c \xi}{2}} \phi \in \ker(\hat{\mathbb{L}}+\lambda \Id),
	\]
	where $\hat{\mathbb{L}}:=-\frac{{\d}^2}{\d\xi^2}+\frac{c^2}{4} \Id-QB$.
\end{prop}

Setting $Q:=\Id$ and following a similar argument as in Proposition \ref{pro:equivalent eigenvalue}, we obtain:

\begin{lem}
	For $\lambda \in \C\backslash\C^-$ and $\phi \in W^{2,2}\left(\R, \mathbb{C}^n\right)$, we have:
	\[
	\phi \in \ker(L-\lambda \Id) \quad \text{if and only if} \quad e^{\frac{c \xi}{2}} \phi \in \ker(\mathbb{L}+\lambda \Id).
	\]
\end{lem}

By the translation invariance property of \eqref{eq:r.d.eq.}, it follows that $\dot{w}^*\in\ker{L}$. Consequently, $e^{\frac{c \xi}{2}}\dot{w}^*\in \ker{\mathbb{L}}$, which implies
\begin{align}\label{eq:L_D cap E^u}
	\begin{pmatrix}
		e^{\frac{c \tau}{2}}\ddot{w}^*(\tau)+\frac{c}{2}e^{\frac{c \tau}{2}}\dot{w}^*(\tau)\\
		e^{\frac{c \tau}{2}}\dot{w}^*(\tau)
	\end{pmatrix}\in E^u(\tau).
\end{align}

\begin{thm}
	If there exists $\xi_0$ such that $\frac{\d}{\d \xi}w^*|_{\xi=\xi_0}=0$, then $w^*$ is spectrally unstable. 
\end{thm}

\begin{proof}
	From the given assumption, together with \eqref{eq:L_D cap E^u} and \eqref{eq:morse index about r.d.eq.}, we conclude
	\[
	\iMor(\mathbb{L})\geq \dim (\Lambda\cap E^u(\tau_0))\geq 1.
	\]
	Combined with Proposition \ref{pro:equivalent eigenvalue}, this implies that at least one positive eigenvalue exists for $L$, completing the proof.
	\end{proof}

	\section{Maslov, Hörmander, Triple Index and Spetral Flow}
	This final section is dedicated to recalling fundamental definitions, key results, and essential properties of the Maslov index and related invariants used throughout our analysis. Primary references include \cite{RS93,HP17,ZWZ18} and their cited works.
	
	\subsection{The Cappell-Lee-Miller Index}\label{sec:maslov}
	Consider the standard symplectic space $(\mathbb{R}^{2n}, \omega)$. Let $\Lag{n}$ denote the Lagrangian Grassmannian of $(\mathbb{R}^{2n}, \omega)$. For $a, b \in \mathbb{R}$ with $a < b$, define $\mathscr{P}([a, b]; \mathbb{R}^{2n})$ as the space of continuous Lagrangian pairs $L: [a, b] \to \Lag{n} \times \Lag{n}$ with compact-open topology. Following \cite{CLM94}, we recall the Maslov index for Lagrangian pairs, denoted by $\iCLM$. Intuitively, for $L = (L_1, L_2) \in \mathscr{P}([a, b]; \mathbb{R}^{2n})$, this index enumerates (with signs and multiplicities) instances $t \in [a, b]$ where $L_1(t) \cap L_2(t) \neq \{0\}$.
	
	\begin{defn}
	The $\iCLM$-index is the unique integer-valued function
	\begin{align*}
	\iCLM: \mathscr{P}([a, b]; \mathbb{R}^{2n}) \ni L \mapsto \iCLM(L(t); t \in [a, b]) \in \mathbb{Z}
	\end{align*}
	satisfying Properties I-VI in \cite[Section 1]{CLM94}.
	\end{defn}
	
	An effective approach to compute the Maslov index employs the crossing form introduced in \cite{RS93}. Let $\Lambda: [0,1] \to \Lag{n}$ be a smooth curve with $\Lambda(0) = \Lambda_0$, and $W$ a fixed Lagrangian complement of $\Lambda(t)$. For $v \in \Lambda_0$ and small $t$, define $w(t) \in W$ via $v + w(t) \in \Lambda(t)$. The quadratic form $Q(v) = \left.\frac{d}{dt}\right|_{t=0} \omega(v, w(t))$ is independent of $W$ \cite{RS93}. A crossing occurs at $t$ where $\Lambda(t)$ intersects $V \in \Lag{n}$ nontrivially. The crossing form at such $t$ is defined as
	\begin{align}\label{eq:def of crossing form in Maslov}
	\Gamma(\Lambda(t), V; t) = \left.Q\right|_{\Lambda(t) \cap V}.
	\end{align}
	
	A crossing is regular if its form is nondegenerate. For quadratic form $Q$, let $\sign(Q) = \coiMor(Q) - \iMor(Q)$ denote its signature. From \cite{ZL99}, if $\Lambda(t)$ has only regular crossings with $V$, then
	\begin{align}\label{eq:compute maslov by cross form}
	\begin{split}
		\iCLM(V, \Lambda(t); t \in [a,b]) = &\coiMor(\Gamma(\Lambda(a), V; a)) \\
	&+ \sum_{a<t<b} \sign \Gamma(\Lambda(t), V; t) - \iMor(\Gamma(\Lambda(b), V; b)).
	\end{split}
	\end{align}
	For the sake of the reader, we list a couple of properties of the $\iCLM$-index that we shall use throughout the paper.
	\begin{itemize}
		\item \textbf{(Reversal)} Let $L:=\left(L_1, L_2\right) \in \mathscr{P}\left([a, b] ; \mathbb{R}^{2 n}\right)$. Denoting by $\widehat{L} \in$ $\mathscr{P}\left([-b,-a] ; \mathbb{R}^{2 n}\right)$ the path traveled in the opposite direction, and by setting $\widehat{L}:=\left(L_1(-s), L_2(-s)\right)$, we obtain
	\begin{align*}
		\iCLM(\widehat{L} ;[-b,-a])=-\iCLM(L ;[a, b])
		\end{align*}
		\item \textbf{(Stratum homotopy relative to the ends)} Given a continuous map
		$L:[a, b] \ni s \rightarrow L(s) \in \mathscr{P}\left([a, b] ; \mathbb{R}^{2 n}\right)$ where $L(s)(t):=\left(L_1(s, t), L_2(s, t)\right)$
		such that $\dim L_1(s, a) \cap L_2(s, a)$ and $\dim L_1(s, b) \cap L_2(s, b)$ are both constant, and then,	
	\begin{align*}
		\iCLM(L(0) ;[a, b])=\iCLM(L(1) ;[a, b])
		\end{align*}
	\end{itemize}

	\subsection{Triple Index and Hörmander Index}
	We summarize key concepts about the triple and Hörmander indices, following \cite{ZWZ18}. For isotropic subspaces $\alpha, \beta, \delta$ in $(\mathbb{R}^{2n}, \omega)$, define the quadratic form
	\begin{align}\label{eq:definition q}
	\mathscr{Q} \coloneqq \mathscr{Q}(\alpha, \beta; \delta): \alpha \cap (\beta + \delta) \to \mathbb{R},\quad \mathscr{Q}(x_1, x_2) = \omega(y_1, z_2)
	\end{align}
	where $x_j = y_j + z_j \in \alpha \cap (\beta + \delta)$ with $y_j \in \beta$, $z_j \in \delta$. For Lagrangian subspaces $\alpha, \beta, \delta$, \cite[Lemma 3.3]{ZWZ18} gives
	\begin{align*}
	\ker \mathscr{Q}(\alpha, \beta; \delta) = \alpha \cap \beta + \alpha \cap \delta.
	\end{align*}
	
	\begin{defn}
	For Lagrangians $\alpha, \beta, \kappa$ in $(\mathbb{R}^{2n}, \omega)$, the triple index is
	\begin{align}\label{eq:def of trip 1}
	\iota(\alpha, \beta, \kappa) = \iMor(\mathscr{Q}(\alpha, \delta; \beta)) + \iMor(\mathscr{Q}(\beta, \delta; \kappa)) - \iMor(\mathscr{Q}(\alpha, \delta; \kappa))
	\end{align}
	where $\delta$ satisfies $\delta \cap \alpha = \delta \cap \beta = \delta \cap \kappa = \{0\}$.
	\end{defn}
	
	By \cite[Lemma 3.13]{ZWZ18}, this index also satisfies
	\begin{align}\label{eq:def of trip 2}
	\iota(\alpha, \beta, \kappa) = \coiMor(\mathscr{Q}(\alpha, \beta; \kappa)) + \dim(\alpha \cap \kappa) - \dim(\alpha \cap \beta \cap \kappa).
	\end{align}
	
	The Hörmander index measures the difference between Maslov indices relative to different Lagrangians. For paths $\Lambda, V \in \mathscr{C}^0([0,1], \Lag{n})$ with endpoints $\Lambda(0)=\Lambda_0$, $\Lambda(1)=\Lambda_1$, $V(0)=V_0$, $V(1)=V_1$:
	
	\begin{defn}\label{def:hormand}
	The Hörmander index is
	\begin{align}\label{eq:def of Hormand}
	\begin{split}
		s(\Lambda_0, \Lambda_1; V_0, V_1) &= \iCLM(V_1, \Lambda(t); t \in [0,1]) - \iCLM(V_0, \Lambda(t); t \in [0,1]) \\
	&= \iCLM(V(t), \Lambda_1; t \in [0,1]) - \iCLM(V(t), \Lambda_0; t \in [0,1]).
	\end{split}
	\end{align}
	\end{defn}
	
	\begin{rem}
	Homotopy invariance ensures Definition \ref{def:hormand} is well-posed (cf. \cite{RS93}).
	\end{rem}
	
	For four Lagrangians $\lambda_1, \lambda_2, \kappa_1, \kappa_2$, \cite[Theorem 1.1]{ZWZ18} establishes:
	\begin{align}\label{eq:relationship Hormand trip}
	s(\lambda_1, \lambda_2; \kappa_1, \kappa_2) = \iota(\lambda_1, \lambda_2, \kappa_2) - \iota(\lambda_1, \lambda_2, \kappa_1) = \iota(\lambda_1, \kappa_1, \kappa_2) - \iota(\lambda_2, \kappa_1, \kappa_2).
	\end{align}
	
	\subsection{Spectral Flow}
	Introduced by Atiyah-Patodi-Singer \cite{APS76}, spectral flow measures eigenvalue crossings. Let $E$ be a real separable Hilbert space, and $\mathscr{CF}^{sa}(E)$ denote closed self-adjoint Fredholm operators with gap topology. For continuous $A: [0,1] \to \mathscr{CF}^{sa}(E)$, the spectral flow $\spfl{A_t; t \in [0,1]}$ counts signed eigenvalue crossings through $-\epsilon$ ($\epsilon > 0$ small).
	
	For each $A_t$, consider the orthogonal decomposition
	\begin{align*}
	E = E_{-}(A_t) \oplus E_0(A_t) \oplus E_{+}(A_t).
	\end{align*}
	Let $P_t$ be the orthogonal projector onto $E_0(A_t)$. At crossing $t_0$ where $E_0(A_{t_0}) \neq \{0\}$, define the crossing form
	\begin{align}\label{eq:def of cross form in sf}
	\operatorname{Cr}[A_{t_0}] \coloneqq P_{t_0} \frac{\partial}{\partial t} P_{t_0}: E_0(A_{t_0}) \to E_0(A_{t_0}).
	\end{align}
	
	A crossing is regular if $\operatorname{Cr}[A_{t_0}]$ is nondegenerate. Define
	\begin{align*}
	\sgn(\operatorname{Cr}[A_{t_0}]) \coloneqq \dim E_{+}(\operatorname{Cr}[A_{t_0}]) - \dim E_{-}(\operatorname{Cr}[A_{t_0}]).
	\end{align*}
	
	Assuming regular crossings, the spectral flow becomes
	\begin{align}\label{eq:compute sf by crossing form}
	\spfl{A_t; t \in [0,1]} = \sum_{t_0 \in \mathcal{S}_*} \sgn(\operatorname{Cr}[A_{t_0}]) - \dim E_{-}(\operatorname{Cr}[A_0]) + \dim E_{+}(\operatorname{Cr}[A_1])
	\end{align}
	where $\mathcal{S}_* = \mathcal{S} \cap (a,b)$ contains crossings in $(a,b)$.
	
	For the sake of the reader we list some properties of the spectral flow that we shall frequently use in the paper.

	\begin{itemize}
		\item \textbf{(Stratum homotopy relative to the ends)} Given a continuous map
		\begin{align*}
		\bar{A}:[0,1] \rightarrow \mathscr{C}^0\left([a, b] ; \mathscr{C} \mathscr{F}^{s a}(E)\right) \text { where } \bar{A}(s)(t):=\bar{A}^s(t)
		\end{align*}	
		such that $\operatorname{dim} \operatorname{ker} \bar{A}^s(a)$ and $\operatorname{dim} \operatorname{ker} \bar{A}^s(b)$ are both independent on $s$, then	
		\begin{align*}
		\operatorname{sf}\left(\bar{A}_t^0 ; t \in[a, b]\right)=\operatorname{sf}\left(\bar{A}_t^1 ; t \in[a, b]\right)
		\end{align*}
		\item \textbf{(Path additivity)} If $A^1, A^2 \in \mathscr{C}^0\left([a, b] ; \mathscr{C} \mathscr{F}^{s a}(E)\right)$ are such that $A^1(b)=A^2(a)$, then
		\begin{align*}
		\operatorname{sf}\left(A_t^1 * A_t^2 ; t \in[a, b]\right)=\operatorname{sf}\left(A_t^1 ; t \in[a, b]\right)+\operatorname{sf}\left(A_t^2 ; t \in[a, b]\right)
		\end{align*}	
		where $*$ denotes the usual catenation between the two paths.
		\item \textbf{(Nullity) }If $A \in \mathscr{C}^0([a, b] ; \operatorname{GL}(E))$, then $\operatorname{sf}\left(A_t ; t \in[a, b]\right)=0$.
	\end{itemize}

	\section*{Acknowledgments}
	The authors thank Professor Xijun Hu for the valuable discussions. The second author gratefully acknowledges the hospitality of the School of Mathematics at Shandong University, where this work was partly carried out during his visit.

%


\vspace{1cm}
\noindent
\textsc{Dr. Ran Yang}\\
School of Science\\
East China  University of Technology\\
Nanchang, Jiangxi, 330013\\
The People's Republic of China \\
E-mail:\ \email{201960124@ecut.edu.cn}

\vspace{1cm}
\noindent
\textsc{Dr. Qin Xing}\\
School of Mathematics and Statistics, Linyi University\\
 Linyi, Shandong \\
The People's Republic of China\\
E-mail:\ \email{xingqin@lyu.edu.cn}

\end{document}